\documentclass{article}
\usepackage[english]{babel}
\newtheorem{definition}{Definition}[section]
\newtheorem{theorem}{Theorem}[section]
\newtheorem{corollary}{Corollary}[theorem]

\newtheorem{remark}{Remark}
\newtheorem{example}{Example}
\newtheorem{proof}{Proof}[section]
\usepackage{amsmath, amssymb, setspace}

\usepackage[english]{babel}
\usepackage{amsfonts}
\usepackage[colorlinks,citecolor=blue,urlcolor=blue,bookmarks=false,hypertexnames=true]{hyperref}
\hypersetup{
    colorlinks=true,
    linkcolor=blue,
    filecolor=magenta,
    urlcolor=cyan,
}
\usepackage{graphics}
\usepackage[numbers,sort&compress]{natbib}
\usepackage{subcaption}
\usepackage{multicol}
\date{}
\usepackage{longtable}
\usepackage{float}
\usepackage[margin=0.7 in]{geometry}
\usepackage{graphicx}
\usepackage{mathtools}
\DeclarePairedDelimiter\norm{\lVert}{\rVert}

\title{Stochastic Processes and Mean Square Calculus on Fractal Curves}
\author{Alireza Khalili Golmankhaneh$^{1}$\footnote{Corresponding author: Alireza Khalili Golmankhaneh.
Email addresses: $alireza.khalili@iau.ac.ir$ or $alirezakhalili2002@yahoo.co.in$ (Alireza Khalili Golmankhaneh ),
$kwelch@ciis.edu$ (Kerri Welch ), $mcserpa@fc.ul.pt$ ( Cristina Serpa )
}, Kerri Welch$^{2}$,  Cristina Serpa$^{3,4,5}$ Ivanka Stamova $^{6}$ \\
	$^{1}$ Department of Physics, Urmia Branch, Islamic Azad University, Urmia 63896,  Iran \\
$^{2}$Faculty at California Institute of Integral Studies, San Francisco, CA, USA \\
$^{3}$ ISEL-Instituto Superior de Engenharia de Lisboa, Portugal\\
$^{4}$ CMAFcIO-Centro de Matem\'{a}tica, Aplica\c{c}\~{o}es Fundamentais e Investiga\c{c}\~{a}o Operacional, Portugal\\
$^{5}$ Instituto Piaget–Cooperativa para o Desenvolvimento Humano, Integral e Ecológico, C.R.L., Portugal\\
$^{6}$ Department of Mathematics, University of Texas at San Antonio, San Antonio, TX 78249, U.S.A.}

\date{\today}
\begin{document}

\maketitle

\begin{abstract}
In this paper, random and stochastic processes are defined on fractal curves. Fractal calculus is used to define cumulative distribution function, probability density function, moments, variance, and correlation function of stochastic process on fractal curves. A new framework which is a generalization of  mean square calculus is formulated.  The sequence of random variables on the fractal curve, fractal mean square  continuity, mean square $F^{\alpha}$-derivative, and fractal mean square integral. The mean square solution of  a fractal stochastic equation  is derived and plotted in order to show the details.

\end{abstract}

\hspace{2cm}~~~~~~ MSC[2010]: 28A80,  60G12, 60G18, 60G20, 60H05

\section{Introduction}
Fractals are shapes that are seen in nature, such as clouds, mountains, coastlines, blood vessels, heart rates, Romanesco broccoli, trees,  frost crystals, and so on \cite{b-1}. Fractal geometry is suggested to characterize the  fractals and  their properties. Fractals often have fractional dimensions and are self-similar and their fractal dimension exceeds their topological dimension \cite{falconer1999techniques}. Analysis on fractals was formulated by many researchers by different methods such as measure theory, harmonic analysis, fractional space, fractional calculus, stochastic process    \cite{freiberg2002harmonic,ma-7,samayoa2022map,ma-12,ma-13,stillinger1977axiomatic,ma-6,ma-8,ma-3}
The Riemann-like method which is base of ordinary calculus, has been modified to include functions with fractal support such as Cantor sets and curves. This framework is simple, algorithmic with geometrical and physical meaning which is called fractal calculus or $F^{\alpha}$-calculus \cite{parvate2009calculus,AD-2,ASAq,satin2013fokker}. Sub-diffusion and super diffusion were modeled by fractal local derivatives  without violating locality and central limit theorem \cite{golmankhaneh2018sub}. Non-local fractal derivatives were defined to model incompressible viscous fluid in fractal media and processes with memory \cite{golmankhaneh2016non,banchuin2022noise}. Fractal calculus has been used to present new model in physics in the fractal space and time \cite{golmankhaneh2021equilibrium}. These models present power law and self-similar solutions and results \cite{BookAlireza}. Fractal calculus was used to find the derivative and integral of the Weierstrass function \cite{gowrisankar2021fractal}. To do research in this direction we have generalized the mean square calculus on fractal curves. \\
The outline of the paper is as follows:\\
Section \ref{1g} we give fractal calculus on curves. We define the stochastic and random variables on fractal curves in Section \ref{2g}. In Section \ref{3g} the fractal mean square calculus is suggested on fractal curves and the corresponding stochastic equation is solved. Finally, Section \ref{4g} is devoted to conclusion.

\section{Preliminaries \label{1g}}
In this section, we  summarize the fractal calculus on fractal curves. \cite{parvate2009calculus,AD-2,ASAq,BookAlireza}.
\subsection{Fractal calculus on fractal curve}
Let the image of the continuous function $f: R\rightarrow R^{n}$ be a fractal.

\begin{definition}
A fractal curve $F\subset R^{n}$ is called  continuously parameterizable if there exists a function such as
\begin{equation}\label{Non-1-1}
  w:[a_{1},b_{1}]\rightarrow F
\end{equation}
where $[a_{1},b_{1}]\subset R$, and $w$ is continuous, one-to-one and onto $F$.
\end{definition}

\begin{definition}\label{FFFas}
A subdivision $P_{[a,b]}$ of interval $[a,b]\subset [a_{1},b_{1}]$ is a finite set of points such as
\begin{equation}\label{Nppo-2}
  P_{[a,b]}=\{a=t_{0},t_{1},...,t_{n}=b\}
\end{equation}
where $t_{i}<t_{i+1}$, and $[t_{i},t_{i+1}]$ is component of the $P_{[a,b]}$. A subdivision $Q$ is called a refinement of $P$ if we have $P\subset Q$.
\end{definition}

\begin{definition}
Let $F$ be a fractal curve, and a subdivision $P_{[a,b]}$ of $[a,b]$,  then  $\sigma^{\alpha}[F,P]$ is defined by
\begin{equation}\label{EWQ-er}
  \sigma^{\alpha}[F,P]=\sum_{i=0}^{n-1}\frac{|w(t_{i+1})-w(t_{i})|^{\alpha}}
  {\Gamma(\alpha+1)}
\end{equation}
where $|*|$ denotes the Euclidean norm on $R^{n}$.
\end{definition}

\begin{definition}
The coarse grained mass $\gamma_{\delta}^{\alpha}(F,a,b)$, for given $\delta$ is defined by
\begin{equation}\label{Non-ij8}
  \gamma_{\delta}^{\alpha}(F,a,b)=\inf_{P_{[a,b]}:|P|\leq \delta} \sigma^{\alpha}[F,P],
\end{equation}
where $|P|=\max_{0\leq i\leq n-1}(t_{i+1}-t_{i})$ for the subdivision $P$.
\end{definition}
\begin{definition}
The mass function $\gamma^{\alpha}(F,a,b)$ is defined by
\begin{equation}\label{REEe787}
  \gamma^{\alpha}(F,a,b)=\lim_{\delta\rightarrow 0} \gamma_{\delta}^{\alpha}(F,a,b).
\end{equation}
The $\gamma^{\alpha}(F,a,b)$ is a monotonic function of $\delta$, and the limit exists.
\end{definition}
\begin{definition}
The $\gamma$-dimension of $F$ is defined by \cite{AD-2,ASAq,BookAlireza}
\begin{equation}
  dim_{\gamma}(F)=\inf\{\alpha:\gamma^{\alpha}(F,a,b)=0\}=
  \sup\{\alpha:\gamma^{\alpha}(F,a,b)=\infty\}.
\end{equation}
\end{definition}
\begin{definition}
The staircase function $S_{F}^{\alpha}:[a_{1},b_{1}]\rightarrow R$ of order $\alpha$ for a $F$ is defined by
\begin{equation}\label{Gttt99}
  S_{F}^{\alpha}(t)=\left\{
                      \begin{array}{ll}
                        \gamma^{\alpha}(F,a,b), & t\geq p_{0} \\
                        -\gamma^{\alpha}(F,a,b), & t< p_{0},
                      \end{array}
                    \right.
\end{equation}
where $p_{0}\in [a_{1},b_{1}]$, and $t\in [a_{1},b_{1}]$. Let consider $F$ for which $  S_{F}^{\alpha}(t)$ be strictly increasing, so it is invertible, and we can write
\begin{equation}\label{Yr-No}
  J(\theta)=S_{F}^{\alpha}(w^{-1}(\theta)),~~~~\theta\in F.
\end{equation}
which is one-to-one.
\end{definition}
\begin{definition}
Let $f:F\rightarrow R$. The limit of $f$ through points of $F$ is $l$, if for given $\epsilon> 0$ there exists $\delta>0$ such that
\begin{equation}\label{nonuu}
  \theta'\in F,~~~and,~~~|\theta'-\theta|<\delta\Rightarrow |f(\theta')-l|<\varepsilon
\end{equation}
and is denoted by
\begin{equation}\label{ETTT}
  l=\underset{\theta'\rightarrow \theta^{-}}{F_{-}lim}f(\theta').
\end{equation}
\end{definition}

\begin{definition}
If $f:F\rightarrow R$, then the $F^{\alpha}$-derivative of function $f$ at $\theta\in F$ is defined by
\begin{equation}\label{Trea223}
  D_{F}^{\alpha} f(\theta)=F-\lim_{\theta'\rightarrow \theta}\frac{f(\theta')-f(\theta)}{J(\theta')-J(\theta)}
\end{equation}
if the limit exists.
\end{definition}

\begin{definition}\label{EQASAaaaa}
 Let us consider finite partition $P_{[a,b]}$ given in Definition \ref{FFFas}. Then, the $F^{\alpha}$-integral of $f:F\rightarrow R$ bounded function on $F$, is defined by
\begin{equation}
\int_{C(a,b)}f(\theta)d_{F}^{\alpha}\theta= \underset{\Delta_{n}\rightarrow 0}{F_{-}lim} \sum_{i=0}^{n}f(\theta_{i'})[S_{F}^{\alpha}(w^{-1}(\theta_{i}))-S_{F}^{\alpha}
(w^{-1}(\theta_{i-1}))]
\end{equation}
where $\theta=w(t)$,~$\theta_{i'}\in[\theta_{i-1},\theta_{i})$, ~$\Delta_{n}=\max_{i}(\theta_{i}-\theta_{i-1})$ and segment $C(t_{1},t_{2})$ is defined by
\begin{equation}\label{qqqqa}
  C(t_{1},t_{2})=\{w(t'):t'\in [t_{1},t_{2}]\}.
\end{equation}
\end{definition}
\section{Random variable on fractal curve \label{2g}}
In this section, we define random variable on fractal curve (RVFC) \cite{khalili2019random,golmankhaneh2021fractalBro,golmankhaneh2020stochastic}. Let consider probability space $(\mathcal{S},\mathcal{F},\mathbb{P})$ where $S$ is the set of outcomes, $ \mathcal{F}$ is the $\sigma$-algebra of events, and $\mathbb{P}$ is the probability measure on sample space $S$ \cite{soong1973random}.
\begin{definition}
A random variable on the fractal curves $F$ (RVFC) is defined by
\begin{equation}
  X(\zeta):\mathcal{S}\rightarrow F
\end{equation}
where $\mathcal{S}$ is sample space.
\end{definition}

\begin{definition}
The cumulative distribution function of RVRC is defined by
\begin{equation}
  F_{X}(\theta)=P(X(\zeta)\leq \theta)~~~\theta\in F.
\end{equation}

\end{definition}

\begin{definition}
The probability density function of RVRC is defined by
\begin{equation}
  f_{X}(\theta)= D_{F}^{\alpha}F_{X}(\theta).
\end{equation}
Then we can write

\begin{equation}
 F_{X}(\theta) =\int_{C(-\infty,t)}f_{X}(\theta)d_{F}^{\alpha}\theta
\end{equation}
where $\theta=w(t)$.
\end{definition}

\begin{definition}\label{DESW45}
The mean of RVRC is defined by
\begin{equation}
  E(X)=\int_{-\infty}^{+\infty}\theta f_{X}(\theta)d_{F}^{\alpha}\theta
\end{equation}
\end{definition}
\begin{definition}
The variance of RVRC  is defined by
\begin{equation}
  Var(X)=\int_{-\infty}^{\infty}(\theta-E(X))^{2}f_{X}(\theta)d_{F}^{\alpha}\theta.
\end{equation}
\end{definition}

\begin{definition}
The $m$-th moment of RVRC is defined by
\begin{equation}
  E(X^{m})=\int_{-\infty}^{+\infty}\theta^{m}f_{X}(\theta)
d_{F}^{\alpha}\theta,~~~m\in N
\end{equation}
\end{definition}

\begin{definition}
A  RVRC is called uniform if its probability density is given by
\begin{equation}
  f_{X}(\theta)=\left\{
                  \begin{array}{ll}
                    \Gamma(\alpha+1), & \theta\in F \\
                    0, & otherwise.
                  \end{array}
                \right.
\end{equation}
and, its  cumulative distribution function is
\begin{equation}
  F_{X}(\theta)=\int_{-\infty}^{\infty}
\Gamma(\alpha+1)d_{F}^{\alpha}\theta=
\left\{
  \begin{array}{ll}
    J(\theta), & \theta \in F \\
    0, & otherwise .
  \end{array}
\right.
\end{equation}

\end{definition}

\begin{definition}
A RVRC is called memoryless if its probability density function is given by
\begin{equation}
  f_{X}(\theta)=\left\{
                  \begin{array}{ll}
                    \lambda \exp(-\lambda\theta), & 0 <\theta, \\
                    0, & \theta<0.
                  \end{array}
                \right.
\end{equation}

Then its cumulative distribution function is
\begin{equation}\label{EDEES}
  F_{X}(\theta)=\left\{
                  \begin{array}{ll}
                    1-\exp(-\lambda J(\theta)), & \theta\geq 0 \\
                    0, &  \theta<0 .
                  \end{array}
                \right.
\end{equation}
where $\lambda>0$.
\begin{figure}[H]
  \centering
  \includegraphics[scale=0.5]{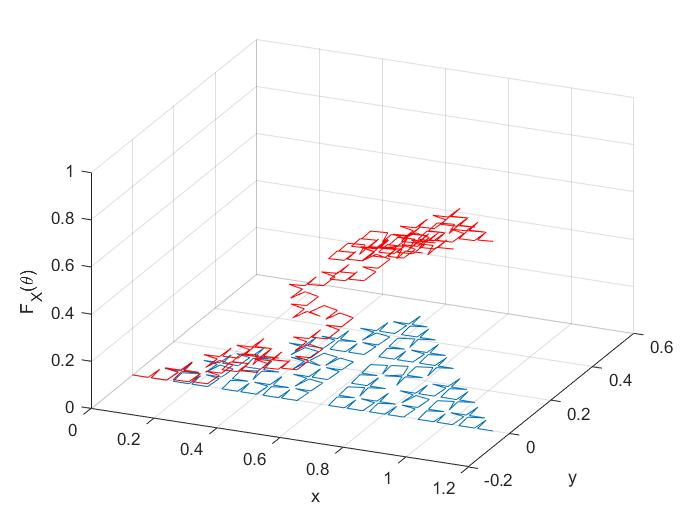}
  \caption{Graph of Eq. \eqref{EDEES} choosing $\lambda=1$}\label{ggrrg2}
\end{figure}

In Figure \ref{ggrrg2}, we have sketched Eq.\eqref{EDEES}.
\end{definition}
\subsection{Random processes on fractal curve}
\begin{definition}
A fractal random process is a family of random variables denoted by
\begin{equation}
  X(\zeta,\tau),~~~~~\zeta\in \mathcal{S},~~~\tau\in F.
\end{equation}
If $\zeta=\zeta_{i}$ is constant, then $X(\zeta_{i},\tau)=X_{i}(\tau)$ is called fractal sample function. For simplicity, let $ X(\zeta,\tau)=X(\tau)$.
\end{definition}

\begin{definition}
The correlation function of a fractal random process $X(\tau)$ is defined by
\begin{equation}\label{ManCo}
  R(\tau_{1},\tau_{2})=E[X(\tau_{1}) X(\tau_{2})],~~~\tau_{1},~\tau_{2} \in F.
\end{equation}
If $\tau_{2}=\tau+\epsilon$, and $\tau_{1}=\tau$, then Eq.\eqref{ManCo} turns to
\begin{equation}
  R(\tau)=E[X(\tau) X(\tau+\epsilon)],~~~\tau \in F.
\end{equation}
\end{definition}

\section{Mean square $F^{\alpha}$-calculus on fractal curves \label{3g}}
In this section, we give below some definitions to develop mean square (f.m.s) calculus \cite{soong1973random} which might be called fractal mean square calculus.
\begin{definition}
The inner product of two random variable on fractal curve $X_{1}$ and $X_{2}$ is defined by
\begin{equation}
  <X_{1},X_{2}>=E[X_{1}X_{2}].
\end{equation}
\end{definition}
\begin{definition}
The distance between two random variable   $X_{1}$ and $X_{2}$ is define by
\begin{equation}
  d(X_{1},X_{2})=||X_{1},X_{2}||
\end{equation}
\end{definition}
\begin{remark}
The class of all second random variable $X_{1},X_{2},...,X_{n}$ on probability space $(\mathcal{S},\mathcal{F},\mathbb{P})$ constitute a linear vector space which is denoted by $L^{\alpha}_{2}$-space.
\end{remark}

\begin{definition}
A random process on fractal curve $X(\tau)$ is called second order if \cite{soong1973random}
\begin{equation}
  ||X(\tau)||=(<X(\tau),X(\tau)>)^{1/2}=R(\tau,\tau)^{1/2}=(E[X(\tau)^{2}])^{1/2}<\infty,~\tau\in F.
\end{equation}
\end{definition}

\begin{definition}
A stochastic process $X(\tau),~\tau\in F$ is a second order stochastic process (s.p.) if, for every set $\tau_{1},\tau_{2},...,\tau_{n}$ are elements of $L_{2}^{\alpha}$-space, namely,
\begin{equation}
 ||X(\tau)||^2= E[X(\tau)^{2}]<\infty,~~~\tau\in F.
\end{equation}

\end{definition}

\begin{definition}
 A sequence of random variable on fractal curve
$X(\tau_{1}), X(\tau_{2}),...,X(\tau_{n})$  is called mean square convergent to $X(\tau)$ if we have
\begin{equation}
  \underset{n \rightarrow \infty}{F_{-}lim}~E[(X(\tau_{n})-X(\tau))^{2}]=0,
\end{equation}
or
\begin{equation}
  \underset{n \rightarrow \infty}{f.l.i.m}~ X(\tau_{n})=X(\tau),
\end{equation}
where the symbol $f.l.i.m$ is  denotes the fractal limit in mean square \cite{soong1973random}.
\end{definition}
\begin{theorem}\label{T431}
  Let $X_{n}$ and $X_{n'}'$ where $n,n'\in N ~(=index ~set)$ be two sequence of second order f.r.v., if
\begin{equation}
  \underset{n \rightarrow n_{0}}{f.l.i.m}~X_{n}=X,~~~\textmd{and}~~~\underset{n' \rightarrow n_{0}'}{f.l.i.m}~X_{n'}'=X'
\end{equation}
where $n_{0}$ and $n_{0}'$ are limit points of $N$. Then
\begin{equation}\label{REDQ}
  \underset{n \rightarrow n_{0},n' \rightarrow n_{0}'}{F_{-}lim}E\{X_{n}X'_{n'}\}=E\{XX'\}.
\end{equation}
\end{theorem}
\begin{proof}
  As $L_{2}^{\alpha}$-space is linear. Then
\begin{equation}
  \Delta X=X_{n}-X,~~~~\Delta X'=X'_{n'}- X',
\end{equation}
are second order f.r.v., and  we have
\begin{align}
 & |E\{X_{n}X'_{n'}-XX'\}|=|E\{X\Delta X'\}+ E\{ X'\Delta X\}+ E\{\Delta X\Delta X'\}|  \nonumber\\&
 \leq ||X||.||\Delta X'||+ ||X'||.||\Delta X||+||\Delta X||.||\Delta X'||,
\end{align}
by using the hypothesis, namely,
\begin{equation}
  \underset{n \rightarrow n_{0}}{F_{-}lim}~||\Delta X||=0,~~~\underset{n' \rightarrow n_{0}'}{F_{-}lim}~||\Delta X'||=0
\end{equation}
we thus get
\begin{equation}
  \underset{n \rightarrow n_{0}, n' \rightarrow n_{0}'}{F_{-}lim}|E\{X_{n}X'_{n'}-XX'\}|=0,
\end{equation}
and the proof is complete.
\end{proof}

\begin{theorem}\label{NBBWQ}
  Let $\{X_{n}(\tau)\},~\tau\in F$ be a sequence of second order. Then it converges to second order process $X(\tau),~\tau\in F$ if, and only if,
\begin{equation}\label{REwaaa}
 \underset{n,n' \rightarrow n_{0}}{F_{-}lim} R_{X_{n}X_{n'}}(\tau,s)=\underset{n,n' \rightarrow n_{0}}{F_{-}lim} E\{X_{n}(\tau),X_{n'}(\tau)\}=R_{XX}(\tau,s)=R(\tau,s),
\end{equation}
 and $R_{XX}(\tau,s)$ is finite on $F\times F$.
\end{theorem}
\begin{proof}
  To establish the if part, set $R(\tau,s)=R_{XX}(\tau,s)$, then
\begin{equation}\label{zfrd}
  ||X_{n}-X_{n'}||^{2}=E\{X_{n}^{2}(\tau)\}-2E\{X_{n}(\tau)X_{n'}(\tau)\}+E\{X_{n'}^{2}(\tau)\}
\end{equation}
It follows that
\begin{equation}
   \underset{n,n' \rightarrow n_{0}}{F_{-}lim}||X_{n}-X_{n'}||^{2}=R(\tau,\tau)-2R(\tau,\tau)+R(\tau,\tau)=0,
\end{equation}
which implies
\begin{equation}
 \underset{n \rightarrow n_{0}}{f.l.i.m}~X_{n}(\tau)= X(\tau).
\end{equation}
To  establish the only if, we make apply of Theorem \ref{T431}. Substituting $X_{n}$ by $X_{n}(\tau)$, $X_{n}'$ by $X_{n}(s)$, $X$ by $X(\tau)$, and $X'$ by $X(s)$, then Eq.\eqref{REDQ} becomes
\begin{equation}\label{MMMM}
  \underset{n \rightarrow n_{0},n' \rightarrow n_{0}' }{F_{-}lim}E[X_{n}(\tau)X_{n'}(s)]=E[X(\tau)X(s)]=R(\tau,s)
\end{equation}
 On setting $\tau=s$, we obtain
\begin{equation}
  \underset{n \rightarrow n_{0},n' \rightarrow n_{0}' }{F_{-}lim}E[X_{n}(\tau)X_{n'}(\tau)]=E[X(\tau)X(\theta)]=R(\tau,\tau)
\end{equation}
which proves only if part. Eq.\eqref{REwaaa} is established by setting $n=n'$ in Eq.\eqref{MMMM}.
\end{proof}

\begin{definition}
 A sequence of random variables on the fractal curve
$X(\tau_{1}), X(\tau_{2}),...,X(\tau_{n})$  is fundamental fractal in probability if, for every $\epsilon>0$
\begin{equation}
  \underset{n,m \rightarrow \infty}{F_{-}lim}~P\big\{|X(\tau_{m})-X(\tau_{n})|>\epsilon\big\}=0,
\end{equation}
or
\begin{equation}
  \underset{n,m \rightarrow \infty}{f.l.i.m}~ X(\tau_{m})=X(\tau_{n}).
\end{equation}
\end{definition}
\begin{definition}
 A sequence of random variables on the fractal curve
$X(\tau_{1}), X(\tau_{2}),...,X(\tau_{n})$  converges in probability   to $X(\tau)$ if we have
\begin{equation}
  \underset{n \rightarrow \infty}{F_{-}lim}~P\big\{|X(\tau_{n})-X(\tau)|>\epsilon\big\}=0,
\end{equation}
or
\begin{equation}
  \underset{n \rightarrow \infty}{f.l.i.m}~ X(\tau_{n})=X(\tau).
\end{equation}
\end{definition}

\begin{definition}
Let $X(\tau)$ be a random process on fractal curve. If
\begin{equation}
  \underset{\epsilon \rightarrow 0}{F_{-}lim}~E[(X(\tau+\epsilon)-X(\tau))^{2}]=0,
\end{equation}
or
\begin{equation}
  \underset{\epsilon \rightarrow 0}{f.l.i.m}~ X(\tau+\epsilon)=X(\tau),
\end{equation}
where $\epsilon>0$ is infinitesimal \cite{keisler2013elementary}, then $X(\tau)$ is called mean square $F^{\alpha}$-continuous at $\tau$.
\end{definition}

\begin{definition}
If a second order f.s.p $X(\tau),~\tau\in F$ is f.m.s continuous at every point $\tau \in C(\tau_{1},\tau_{2})$, then $X(\tau)$ is f.m.s. continuous on segment $C(\tau_{1},\tau_{2})$.
\end{definition}

\begin{definition}\label{defsecond}
A random process $X(\tau)$ have mean square $F^{\alpha}$-derivative if we have
\begin{equation}
  \underset{\epsilon \rightarrow 0}{F_{-}lim}~E\bigg[\bigg(\frac{X(\tau+\epsilon)-X(\tau)}{J(\epsilon)}-
D_{F}^{\alpha}X(\tau)\bigg)^{2}\bigg]=0,
\end{equation}
or
\begin{equation}
  \underset{\epsilon \rightarrow 0}{f.l.i.m}~\frac{X(\tau+\epsilon)-X(\tau)}{J(\epsilon)}=
D_{F}^{\alpha}X(\tau).
\end{equation}
\end{definition}
\begin{definition}
If a second order f.s.p $X(\tau),~\tau\in F$ is f.m.s differentiable at every $\tau\in C(\tau_{1},\tau_{2})\subset F$ then $X(\tau)$ is f.m.s differentiable on $ C(\tau_{1},\tau_{2})$.
\end{definition}
\begin{example}
Consider the fractal stochastic process as
\begin{equation}
  X(\tau)=\cos(\tau+\Phi)
\end{equation}
where $\Phi$ is uniformly distributed on $F$. Its correlation function is
\begin{equation}\label{TGfr}
  R(\tau)=\frac{1}{2}\cos(\tau).
\end{equation}
As Eq.\eqref{TGfr} is $F^{\alpha}$-differentiable infinitely hence $X(\tau)$ has f.m.s derivative.
\end{example}
\begin{theorem}
If $X(\tau)$ and $Y(\tau)$ are f.m.s. differentiable at $\tau$, then we have
\begin{equation}
D_{F}^{\alpha}[aX(\tau)+bY(\tau)]=a D_{F}^{\alpha}X(\tau)+b D_{F}^{\alpha}Y(\tau)
\end{equation}
where $a$ and $b$ are constant.
\end{theorem}
\begin{proof}
 By the norm property, the proof is immediate as
\begin{align}
&\norm[\bigg]{\frac{aX(\tau+\epsilon)+bY(\tau +\epsilon)-aX(\tau)-bX(\tau)}{\epsilon}
-aD_{F}^{\alpha}X(\tau)+b D_{F}^{\alpha}Y(\tau)}\nonumber\\&
\leq \norm[\bigg]{a\bigg[\frac{X(\tau+\epsilon)-X(\tau)}{\tau}-
D_{F}^{\alpha}X(\tau)\bigg]}+
\norm[\bigg]{b\bigg[\frac{Y(\tau +\epsilon)-Y(\tau)}{\epsilon}-D_{F}^{\alpha}Y(\tau)\bigg]},
\end{align}
by hypothesis the last two terms tend to zero as $\epsilon\rightarrow 0$, which completes the proof.
\end{proof}
\begin{theorem}
If $f: F\rightarrow \mathbb{R}$ and $F^{\alpha}$-differentiable at
 $\tau\in F$, and $X(\tau)$ is f.m.s differentiable at $\tau \in F$, then $f(\tau)$, $X(\tau)$ is f.m.s differentiable at $\tau$ and we have
\begin{equation}
  D_{F}^{\alpha}[f(\tau)X(\tau)]=D_{F}^{\alpha}f(\tau) X(\tau)+ f(\tau) D_{F}^{\alpha}X(\tau).
\end{equation}
\end{theorem}
\begin{proof}
To show this property, consider the norm property as
\begin{align}\label{REWQu}
&\norm[\bigg]{\frac{f(\tau+\epsilon)X(\tau+\epsilon)-f(\tau)X(\tau)}{\tau}-
D_{F}^{\alpha}f(\tau)X(\tau)-f(\tau)D_{F}^{\alpha}X(\tau)}\nonumber\\&
\leq \norm[\bigg]{\frac{f(\tau+\epsilon)X(\tau+\epsilon)-f(\tau)X(\tau+\epsilon)}{\tau}-
D_{F}^{\alpha}f(\tau)X(\tau)}\nonumber\\&
+\norm[\bigg]{\frac{f(\tau)X(\tau+\epsilon)-f(\tau)X(\tau)}{\tau}-f(\tau)D_{F}^{\alpha}X(\tau)}\nonumber\\&
\leq \norm[\bigg]{\bigg[\frac{f(\tau+\epsilon)-f(\tau)}{\epsilon}-D_{F}^{\alpha}
f(\tau)\bigg]X(\tau+\epsilon)}
+\norm[\bigg]{D_{F}^{\alpha}f(\tau)\bigg[X(\tau+\epsilon)-X(\epsilon)\bigg]}\nonumber\\&
+\norm[\bigg]{f(\tau)\bigg[\frac{X(\tau+\epsilon)-X(\tau)}{\epsilon}-D_{F}^{\alpha}f(\tau)]}\nonumber\\&
\leq \bigg|\frac{f(\tau-\epsilon)-f(\tau)}{\epsilon}-D_{F}^{\alpha}f(\tau)\bigg|
\norm[\bigg]{X(\tau+\epsilon)}+\bigg|D_{F}^{\alpha}f(\tau)\bigg|
\norm[\bigg]{X(\tau+\epsilon)-X(\tau)}\nonumber\\&
+\bigg|f(t)\bigg|\norm[\bigg]{\frac{X(\tau+\epsilon)-X(\tau)}{\epsilon}-D_{F}^{\alpha}X(\tau)}
\end{align}
The Eq.\eqref{REWQu} tends to zero as $\epsilon\rightarrow 0$. Thus the proof is complete.
\end{proof}
\begin{theorem}
If a s.p. $X(\tau),\tau\in F$ is $n$ times f.m.s differentiable then we have
\begin{equation}\label{uuyRe}
E\{D_{F}^{\alpha}X(\tau)\}=D_{F}^{\alpha}E\{X(\tau)\}.
\end{equation}
\end{theorem}
\begin{proof}
To establish Eq.\eqref{uuyRe}, let recall Definition \eqref{defsecond}, then we have
\begin{align}
  &E\{D_{F}^{\alpha}X(\tau)\}=E\bigg\{\underset{\epsilon \rightarrow 0}{f.l.i.m} \bigg[ \frac{X(\tau+\epsilon)-X(\tau)}{\epsilon}\bigg]\bigg\}\nonumber\\&
=\underset{\epsilon \rightarrow 0}{F_{-}lim}\bigg\{\frac{E[X(\tau+\epsilon)]-E[X(\tau)]}{\epsilon}\bigg\}\nonumber\\&
=D_{F}^{\alpha}E[X(\tau)],
\end{align}
which completes the proof.
\end{proof}

\begin{definition}
The second fractal generalized derivative of a second order s.p.
$X(\tau),\tau \in F$ is defined by
\begin{align}
  &\underset{\epsilon, \epsilon'\rightarrow 0 }{F_{-}lim}\bigg(\frac{1}{\epsilon \epsilon'}\bigg)\Delta_{\epsilon}\Delta_{\epsilon'}R(\tau,s)\nonumber\\&=
\underset{\epsilon, \epsilon'\rightarrow 0 }{F_{-}lim}\bigg(\frac{1}{\epsilon \epsilon'}\bigg)
\bigg[R(\tau+\epsilon,s+\epsilon')-R(\tau+\epsilon,s)
-R(\tau,s+\epsilon')+R(\tau,s)\bigg]
\end{align}
if it exists at $(\tau,\tau)$ and is finite.
\end{definition}
\begin{theorem}\label{REDWa}
If a s.p. $X(\tau)$ is mean square $F^{\alpha}$-differentiable at point $\tau$, then $X(\tau)$ is f.m.s continuous at $\tau$.
\end{theorem}
\begin{proof}
Since  $\tau+\epsilon\in F$, then we write
\begin{align}
&  \underset{\epsilon \rightarrow 0}{F_{-}lim}~\norm[\bigg]{X(\tau+\epsilon)-X(\tau)}^{2}=\underset{\epsilon \rightarrow 0}{F_{-}lim}~ |\epsilon|^{2}\norm[\bigg]{\frac{[X(\tau+\epsilon)-X(\tau)]}{\epsilon}}^{2}\nonumber\\&=
0\times~\underset{\epsilon \rightarrow 0}{F_{-}lim}~ \bigg(\frac{1}{\epsilon^{2}}\bigg)\Delta_{\epsilon}\Delta_{\epsilon} R(\tau,\tau)=0.
\end{align}
This is the desired conclusion.
\end{proof}

\begin{example}
 Consider a s.p. $X(\tau),\tau\in F$ as
\begin{equation}
  X(\tau)=A\tau,
\end{equation}
where $A$ is a second-order fractal r.v. with mean zero and variance $\sigma^2$. Its correlation function is
\begin{equation}
  R(\tau,s)=\sigma^{2}\tau s.
\end{equation}
The second fractal generalized  derivative of $X(\tau)$ is
\begin{align}
&\underset{\epsilon, \epsilon'\rightarrow 0 }{F_{-}lim}\bigg(\frac{1}{\epsilon \epsilon'}\bigg)\Delta_{\tau}\Delta_{\tau'}R(\tau,s)\nonumber\\&=
\underset{\epsilon, \epsilon'\rightarrow 0 }{F_{-}lim}\bigg(\frac{\sigma^{2}}{\epsilon \epsilon'}\bigg)\bigg[(\tau+\epsilon)(s+\epsilon')-(\tau+\epsilon)s-\tau(s+\tau')
+\tau s \bigg]\nonumber\\&=
\underset{\epsilon, \epsilon'\rightarrow 0 }{F_{-}lim}\bigg(\frac{\sigma^{2}}{\epsilon \epsilon'}\bigg)[\epsilon \epsilon']=\sigma^2<\infty,
\end{align}
which implies that $X(\tau)$ is f.m.s differentiable at every finite $\tau$.
\end{example}
\begin{theorem}
If the second fractal generalized  derivative exist at $(\tau,\tau)$ for every $\tau\in F$. Then the partial derivative
\begin{equation}
  D_{F,\tau}^{\alpha} R(\tau,s),~D_{F,s}^{\alpha} R(\tau,s),~D_{F,\tau}^{\alpha}D_{F,s}^{\alpha}R(\tau,s)
\end{equation}
exists and is finite on $F\times F$.
\end{theorem}
\begin{proof}
  Since $X(\tau)$ is f.m.s differentiable, thus we can write
\begin{align}\label{EWqaq2Wq}
 & E\{D_{F,\tau}^{\alpha} X(\tau)X(s)\}=E\bigg\{\underset{\epsilon \rightarrow 0 }{f.l.i.m} \bigg[\frac{X(\tau+\epsilon)-X(\tau)}{\epsilon}\bigg]
X(s)\bigg\}\nonumber\\&
=\underset{\epsilon \rightarrow 0 }{F_{-}lim}~ \bigg(\frac{E\{X(\tau+\epsilon)X(s)-X(\tau)X(s)\}}{\epsilon}\bigg)\nonumber\\&
=\underset{\epsilon \rightarrow 0 }{F_{-}lim} ~[R(\tau+\epsilon,s)-R(\tau,s)]\nonumber\\&
=D_{F,\tau}^{\alpha} R(\tau,s).
\end{align}
Likewise, one can establish that
\begin{equation}\label{EWqaq2}
  E\{ X(\tau)D_{F,s}^{\alpha}X(s)\}=D_{F,s}^{\alpha} R(\tau,s).
\end{equation}
Both Eqs. \eqref{EWqaq2Wq} and \eqref{EWqaq2} exist and are finite on $F\times F$. Finally, we can write
\begin{align}
  &E\{D_{F,\tau}^{\alpha} X(\tau)D_{F,s}^{\alpha}X(s)\}= E\bigg\{\underset{\epsilon,\epsilon' \rightarrow 0 }{f.l.i.m}\bigg[
\frac{X(\tau+\epsilon)-X(\tau)}{\epsilon}\bigg]
\bigg[\frac{X(s+\epsilon')-X(s)}{\epsilon}\bigg]\bigg\}\nonumber\\&
=\underset{\epsilon,\epsilon' \rightarrow 0 }{F_{-}lim}~\frac{1}{\epsilon}~E\bigg\{
\bigg[\frac{X(\tau+\epsilon)X(\tau+\epsilon')-X(\tau+\epsilon)X(s)}{\epsilon'}\bigg]-
\bigg[\frac{X(\tau)X(\tau+\epsilon')-X(\tau)X(s)}{\epsilon'}\bigg]\bigg\}\nonumber\\&
=\underset{\epsilon \rightarrow 0 }{F_{-}lim}~\frac{1}{\epsilon}~\underset{\epsilon' \rightarrow 0 }{F_{-}lim}~\bigg
\{\frac{R(\tau+\epsilon,s+\epsilon')-R(\tau+\epsilon,s)}{\epsilon'}-
\frac{R(\tau,s+\epsilon')-R(\tau,s)}{\epsilon'}\bigg\}\nonumber\\&
=\underset{\epsilon \rightarrow 0 }{F_{-}lim}~\frac{1}{\epsilon}~[D_{F,s}^{\alpha} R(\tau+\epsilon,s)-D_{F,s}^{\alpha} R(\tau,s)]=D_{F,\tau}^{\alpha}D_{F,s}^{\alpha}R(\tau,s),
\end{align}
which exists and is finite on $F\times F$ and completes proof.
\end{proof}
\begin{definition}
A second f.s.p $X(\tau)$ is f.m.s analytic on $F$, if it can be expanded in the f.m.s. convergent Taylor series as
\begin{equation}
  X(\tau)=\sum_{n=0}^{\infty}\frac{(J(\tau)-J(\tau_{0}))^{n}}{n!}
D_{F,\tau}^{n\alpha} X(\tau)\bigg|_{\tau=\tau_{0}},~~~\tau,\tau_{0}\in F.
\end{equation}
\end{definition}
\begin{definition}
 Let $P_{[a,b]}$ be a finite subdivision which is given in Definition \ref{EQASAaaaa},  $X(\tau),~\tau\in F$ be a second order f.s.p. on $[a,b]\subset F$, and $f(\tau,u)$ be a function on $\tau\in [a,b]$. For every $u: F\rightarrow \mathbb{R}$ function, we form the random variable as
\begin{equation}
  Y_{n}(u)=\sum_{i=1}^{n}f(\tau_{i}',u)X(\tau_{i}')(J(\tau_{i})-J(\tau_{i-1}))
\end{equation}
Since $L_{2}^{\alpha}$-space is linear, so we have $Y_{n}(u)\in L_{2}^{\alpha}$.
\end{definition}
\begin{definition}
  If for every $u$
\begin{equation}
 \underset{n\rightarrow \infty  \Delta_{n} \rightarrow0}{f.l.i.m} Y_{n}(u)=Y(u),
\end{equation}
exists for some sequence of subdivision $P$ defined in Definition \ref{EQASAaaaa}, then the $ Y(u)$ is called the definite fractal mean square integral (f.m.s.i) of $f(\tau,u)X(\tau)$ over the interval $[a,b]$ or mean square $F^{\alpha}$-integral  and it is denoted by
\begin{equation}\label{RRRWWAAA}
  Y(u)=\int_{C(a,b)}f(\tau,u)X(\tau)d_{F}^{\alpha}\tau.
\end{equation}
\end{definition}

\begin{theorem}
The f.s.p $Y(u)$ defined by Eq.\eqref{RRRWWAAA} exists if and only if the fractal double integral
\begin{equation}
  \int_{C(a,b)}\int_{C(a,b)} f(\tau,u)f(s,u) R(\tau,s)d_{F}^{\alpha}\tau d_{F}^{\alpha}s
\end{equation}
exists and is finite.
\end{theorem}
\begin{proof}
This theorem is an immediate result of the convergence in fractal mean square Theorem \ref{NBBWQ}. Namely, if we choose
\begin{equation}
  X_{n}(\tau)=\sum_{i=1}^{n}f(\tau_{i}',\tau)X(\tau_{i}')(J(\tau_{i})-J(\tau_{i-1})),~~
\tau=u,~~n_{0}=\infty.
\end{equation}
\end{proof}
\begin{definition}
An improper f.m.s integral is defined by
\begin{equation}
  \int_{C(a,\infty)}f(\tau,u)X(\tau)d_{F}^{\alpha}\tau=\underset{b \rightarrow\infty }{f.l.i.m}\int_{C(a,b)}f(\tau,u)X(\tau)d_{F}^{\alpha}\tau.
\end{equation}
It is easy to see that it exists if, and only if, the improper fractal double integral
\begin{equation}
  \int_{C(a,\infty)}\int_{C(a,\infty)}f(\tau,u)f(s,u)R(\tau,s)d_{F}^{\alpha}\tau d_{F}^{\alpha}s=\underset{b \rightarrow \infty }{F_{-}lim}\int_{C(a,b)}\int_{C(a,b)}f(\tau,u)f(s,u)R(\tau,s)d_{F}^{\alpha}\tau d_{F}^{\alpha}s,
\end{equation}
exists and is finite.
\end{definition}
\begin{remark}
The mean square integral of f.s.p $X(\tau),~\tau\in F$  is defined by using Eq.\eqref{RRRWWAAA} and setting $f(\tau,u)=1$ as follows:
\begin{align}
  Z(\tau)&=\int_{C(t_{0},\tau)}X(\theta)d_{F}^{\alpha}\theta\\
&=\underset{\Delta^{\alpha}J(\theta_{i}),\Delta^{\alpha}J(\theta_{k})  \rightarrow 0}{F_{-}lim}~E\bigg\{\bigg[\sum_{i}X(\theta_{i})\Delta^{\alpha}J(\theta_{i})-\sum_{k}
X(\theta_{k})\Delta^{\alpha}J(\theta_{k})\bigg]^{2}\bigg\},
\end{align}
where $\theta_{0}<\theta_{1}<...<\theta$ ,~$\theta_{i}=w(t_{i})$,  $\Delta^{\alpha}J(\theta_{k})=J(\theta_{k})-J(\theta_{k-1})$ and $\Delta^{\alpha}J(\theta_{i})=J(\theta_{i})-J(\theta_{i-1})$.
\end{remark}
\begin{theorem}
If $X(\tau)$ is f.m.s continuous on $[a,\tau]\subset F$, then
\begin{equation}
  Y(\tau)=\int_{C(a,\tau)}X(s)d_{F}^{\alpha}s
\end{equation}
is f.m.s continuous on $[a,\tau]$. It is also f.m.s differentiable on $[a,\tau]$ with
\begin{equation}
  D_{F,\tau}^{\alpha}Y(\tau)=X(\tau).
\end{equation}
\end{theorem}
\begin{proof}
  We only prove the second part on $F^{\alpha}$-differentiability, since the first part is evident using Theorem \ref{REDWa}. Let us consider
\begin{align}
 & \norm[\bigg]{\frac{1}{\epsilon}\bigg[\int_{C(a,\tau+\epsilon)}X(s)d_{F}^{\alpha}s-
\int_{C(a,\tau)}X(s)d_{F}^{\alpha}s\bigg]-X(\tau)}\nonumber\\&=
\norm[\bigg]{\frac{1}{\epsilon}\int_{C(a,\tau+\epsilon)}
[X(s)-X(\tau)]d_{F}^{\alpha}s}\nonumber\\& \leq
|\frac{1}{\epsilon}|\int_{C(a,\tau+\epsilon)}
\norm[\bigg]{[X(s)-X(\tau)]}d_{F}^{\alpha}s\nonumber\\& \leq
\max_{s\in [\tau,\tau+\epsilon]} \norm[\bigg]{X(s)-X(\tau)}
\end{align}
as $\epsilon\rightarrow 0$ from the hypothesis, we arrive at the result.
\end{proof}
\begin{corollary}
If $X(\tau)$ is f.m.s integrable on $F$ and $f(\tau,s)$ if $F$-continuous on $F\times F$ with a finite first partial derivative
$D_{F,\tau}^{\alpha} f(\tau,s)$, then the f.m.s derivative of
\begin{equation}
  Y(\tau)=\int_{C(a,\tau)}f(\tau,s)X(s)d_{F}^{\alpha}s
\end{equation}
exists at all $\tau\in F$, and
\begin{equation}
 D_{F,\tau}^{\alpha}Y(\tau)=\int_{C(a,\tau)} D_{F,\tau}^{\alpha}f(\tau,s) X(s)d_{F}^{\alpha}s+f(\tau,\tau)X(\tau).
\end{equation}
\end{corollary}
The proof is obvious. This result may called the fractal mean square counterpart of the Leibniz rule in ordinary calculus.

\begin{theorem}
  If $X(\tau)$ be m.s $F^{\alpha}$-differentiable on $F$, and let $f:F\rightarrow \mathbb{R}$ be $F$-continuous on $F\times F$, whose partial derivative $D_{F,s}^{\alpha} f(\tau,s)$ exists. If
\begin{equation}
  Y(\tau)=\int_{C(a,\tau)}f(\tau,s)D_{F,s}^{\alpha}X(s)d_{F}^{\alpha}s,
\end{equation}
then
\begin{equation}
  Y(\tau)=f(\tau,s)X(s)\bigg|_{a}^{\tau}-
\int_{C(a,\tau)}D_{F,s}^{\alpha}f(\tau,s)X(s)d_{F}^{\alpha}s.
\end{equation}
\end{theorem}
 The proof is straightforward. This property is fractal mean square counter part of fundamental theorem of ordinary calculus and may called the fundamental theorem of fractal mean square calculus.

\begin{example}
Let us consider stochastic second order mean square on fractal curve as
\begin{equation}\label{R58FE}
  (D_{F}^{\alpha})^{2}X(\tau)+A^{2}X(\tau)=0,~~~~X(\tau)|_{\tau=0}=X_{0},~~ \textmd{and} ~~~D_{F}^{\alpha}X(\tau)|_{\tau=0}=X_{1}.
\end{equation}
where $A^2$ is  Beta random variable $Beta(\mu,\nu)$, on fractal curve, namely,
\begin{equation}
  E[A^2]=\frac{\mu}{\mu+\nu},~~~~~Var[A^2]=\frac{\mu\nu}{(\mu+\nu)^2(\mu+\nu+1)}
\end{equation}
and it is independent from $X_{0}$ and $X_{1}$. To find solution of Eq.\eqref{R58FE}, we use Frobenius like-method and consider an infinite series solution as
\begin{equation}\label{ZAQ963}
  X(\tau)=\sum_{m=0}X_{m}J(\tau)^{m},
\end{equation}
where $X_{m}$ are random variables on fractal curves. The f.m.s derivative from both sides of Eq.\eqref{ZAQ963} and a computation gives
\begin{equation}\label{QQQQQ}
  (D_{F}^{\alpha})^{2}X(\tau)=\sum_{m=0}(m+2)(m+1)X_{m+2}J(\tau)^{m}.
\end{equation}
By replacing Eq.\eqref{ZAQ963} and Eq.\eqref{QQQQQ} into Eq.\eqref{R58FE} we obtain
\begin{equation}
\sum_{m=0}(m+2)(m+1)X_{m+2}+A^{2}X_{m}J(\tau)^{m}=0
\end{equation}
It follows that
\begin{equation}
  X_{m+2}=-\frac{A^{2}X_{m}}{(m+2)(m+1)}.
\end{equation}
By some manipulations, and using initial conditions, we get the solution as
\begin{equation}\label{jhhu}
  X(\tau)=X_{0}\cos(A J(\tau))+\frac{X_{1}}{A}\sin(A J(\tau))
\end{equation}
where
\begin{align}
  \cos(A J(\tau))&=\sum_{m=0}\frac{(-1)^{m}A^{2m}}{(2m)!}J(\tau)^{2m},\nonumber\\~~~\sin(A J(\tau))&=\sum_{m=0}\frac{(-1)^{m}A^{2m+1}}{(2m+1)!}J(\tau)^{2m+1}
\end{align}
\end{example}
The truncated Taylor series of solution Eq.\eqref{jhhu} is
\begin{equation}
  X_{N}(\tau)=X_{0}\sum_{m=0}^{N}\frac{(-1)^{m}A^{2m}}{(2m)!}
J(\tau)^{2m}+
X_{1}\sum_{m=0}^{N}\frac{(-1)^{m}A^{2m}}{(2m+1)!}J(\tau)^{2m+1}
\end{equation}
The mean of $X_{N}(\tau)$ by Definition \ref{DESW45} is
\begin{equation}\label{ESwqa}
  E[X_{N}(\tau)]=E[X_{0}]\sum_{m=0}^{N}\frac{(-1)^{m}E[A^{2m}]}{(2m)!}
J(\tau)^{2m}+
E[X_{1}]\sum_{m=0}^{N}\frac{(-1)^{m}E[A^{2m}]}{(2m+1)!}J(\tau)^{2m+1}.
\end{equation}
To derive variance of $X(\tau)$, let us first calculate
\begin{align}\label{Qawq47}
  E[X_{N}(\tau)^{2}]&=\sum_{m=0}^{N}\bigg(\frac{E[X_{0}]^{2}E[(A^{2})^{2m}]}{((2m)!)^2}J(\tau)^{4m}
+\frac{E[X_{1}^2]E[(A^{2})^{2m}]}{((2m+1)!)^2}\bigg)J(\tau)^{4m+2}\nonumber\\&+
2E[X_{0}X_{1}]\sum_{n=0}^{N}\sum_{m=0}^{N}\frac{(-1)^{n+m}E[(A^2)^{n+m}]}
{(2n)!(2m+1)!}J(\tau)^{2(n+m)+1}
\end{align}
By approximating  Eqs.\eqref{ESwqa} and \eqref{Qawq47}, we arrive at
\begin{align}
   E[X_{N}(\tau)]&=E[X_{0}]-\frac{E[A^2]}{2}J(\tau)^2+E[X_{1}]J(\tau)-
\frac{E[A^{2}]}{3!}J(\tau)^{3}+\frac{E[A^{4}]}{4!}J(\tau)^4+
\frac{E[A^{4}]}{5!}J(\tau)^5+...\nonumber\\
E[X_{N}(\tau)^{2}]&=E[X_{0}^2]+2E[X_{0}X_{1}]J(\tau)+E[X_{1}^{2}]J(\tau)^{2}+...
\end{align}
If we choose $E[X_{0}]=1,~E[A^2]=2/3,~E[X_{1}]=1,~E[X_{0}^2],~E[X_{1}^2]=1,~E[X_{0}X_{1}]=1$, we have
\begin{align}\label{xuuu}
   E[X_{N}(\tau)]&=1+J(\tau)-\frac{1}{3}J(\tau)^{2}-\frac{1}{9}J(\tau)^{3}+...\nonumber\\
E[X_{N}(\tau)^{2}]&=1+2J(\tau)+J(\tau)^{2}+...
\end{align}
\begin{figure}[H]
  \centering
  \includegraphics[scale=0.5]{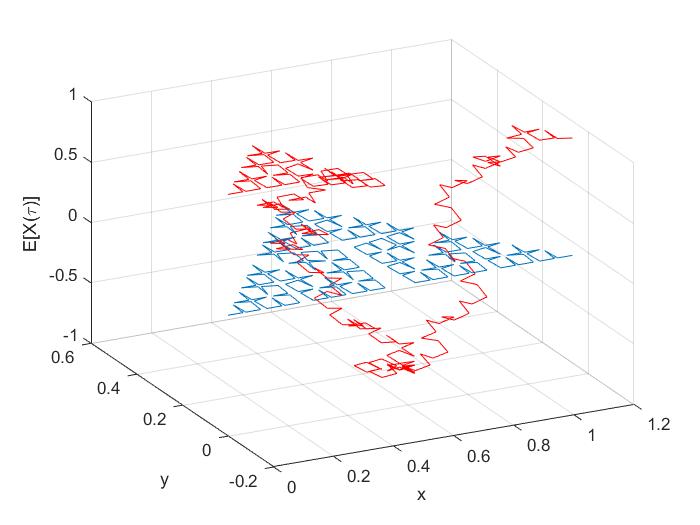}
  \caption{Graph of Eq. \eqref{ESwqa} choosing $E[X_{1}]=0, E[X_{0}]=1, E[A^2]=4$ }\label{ggg2}
\end{figure}
In Figure \ref{ggg2}, we have plotted Eq. \eqref{ESwqa}.\\
\textbf{Acknowledgment } Cristina Serpa acknowledges partial funding by national funds through FCT - Foundation for Science and Technology, project reference: UIDB/04561/2020.

\section{Conclusion \label{4g}}

In this work, we have generalized random variables and processed them on fractal curves by defining mean, variance, and correlation functions. The  second random variable,  the mean square convergent, fundamental fractal in probability, converges in probability, mean square $F^{\alpha}$-continuous, mean square $F^{\alpha}$-derivative, and mean square $F^{\alpha}$-integral are defined to formulate a new framework on fractal curves. This framework is based on new mathematical models for science and physics to apply processes with fractal structure.

\bibliographystyle{unsrt}
\bibliography{HPM222}

\end{document}